\documentclass[12pt,english]{amsart}
                                                        
\usepackage[inner=1in,outer=1in,top=1in,bottom=1in, marginparwidth=.7in,marginparsep=.1in]{geometry}
\usepackage{mathtools}

\usepackage[T1]{fontenc}
\usepackage{amssymb}
\usepackage[all]{xy}
\usepackage{colordvi}
\usepackage{graphicx}

\usepackage{paralist}
\usepackage{babel}
\usepackage{amsmath}
\usepackage{amsthm}
\usepackage{tikz}

\newcommand\tx{{\tilde x}}
\newcommand\tX{{\tilde X}}
\newcommand\tI{{\tilde I}}

\newcommand\C{\mathbb{C}}

\renewcommand\P{\mathbb{P}}

\newcommand\R{\mathbb{R}}

\newcommand\Z{\mathbb{Z}}

\newcommand\hx{\widehat{x}}

\theoremstyle{plain}
\newtheorem{Theorem}{Theorem}[section]
\newtheorem{Proposition}[Theorem]{Proposition}
\newtheorem{Corollary}[Theorem]{Corollary}
\newtheorem{Lemma}[Theorem]{Lemma}

\newtheorem*{Conjecture4M-4}{The $\mathbf{4M-4}$ Conjecture \cite{BCMN:14}}

\theoremstyle{definition}

\newtheorem{Example}[Theorem]{Example}
\newtheorem{Remark}[Theorem]{Remark}

\newtheorem*{Constr*}{Construction}
\newtheorem*{Def*}{Definition}
\newtheorem*{Example*}{Example}
\newtheorem*{Remark*}{Remark}
\begin{document}
\title[The geometry of ambiguity in 1D phase retrieval]{The geometry of ambiguity in one-dimensional phase retrieval}
\author{Dan Edidin}
\address{Department of Mathematics, University of Missouri-Columbia, Columbia, Missouri 65211}
\email{edidind@missouri.edu}
\thanks{The author's research was supported in part by  Simons collaboration grant 315460.}
\date{\today}

\begin{abstract}
We consider the geometry associated to
the ambiguities of the one-dimensional Fourier phase retrieval problem
for vectors in $\C^{N+1}$.
Our first result states that
the space of signals has a finite covering (which we call
the {\em root covering}) where any two signals in the covering space with
the same Fourier intensity function differ by a {\em trivial covering
  ambiguity}.

Next we use the root covering to study how the non-trivial ambiguities of
a signal vary as the signal varies. This is done by describing the
incidence variety of pairs of signals with same Fourier intensity
function modulo global phase.  As an application we give a criterion
for a real subvariety of the space of signals to admit generic phase
retrieval. The extension of this result to multi-vectors played an
important role in the author's work with Bendory and Eldar on blind
phaseless short-time Fourier transform recovery.
\end{abstract}

\maketitle

\maketitle

  \section{Introduction}

  The Fourier transform of a vector $x \in \C^{N+1}$ is the
  is the
  polynomial $S^1 \to \C$ defined by the formula
  $\displaystyle{\hat{x}(\omega) = \sum_{n=0}^N x[n]\omega^{n}}$ where
  we take $\omega = e^{-\iota \theta}$ to be a coordinate on the unit
  circle. Clearly, any vector is uniquely determined by its Fourier transform.
  
The phase retrieval problem asks if it possible to uniquely recover a vector
  $x \in \C^{N+1}$ from its Fourier intensity function
  $A(\omega) = |\hat{x}(\omega)|^2$. 
This problem occurs in many indirect measurement systems including crystallography, astronomy and optics.
For a contemporary review of phase retrieval in optical imaging see \cite{SECCMS:15}.

  Unfortunately, this problem is ill-posed. Obviously,  $x$ and $e^{\iota \theta} x$
  have the same Fourier intensity function, as does the vector
  $\dot{x}$ obtained by reflection
and conjugation since
  $\hat{\dot{x}} = \overline{\hat{x}}$. However, even modulo these
  {\em trivial ambiguities} the phase retrieval problem has no
  unique solution \cite{BrSo:79,OpSc:10}. In fact, it is known that for given
  $x$ there are up to $2^{N-1}$ vectors modulo trivial ambiguities
  with the same Fourier intensity function. These vectors are referred
  to as the {\em non-trivial ambiguities} of the phase retrieval
  problem \cite{BePl:15, BePl:18}.

  \begin{Example} \label{ex.ambig}
    Let $x = (9/2, -9,-1/2,1)$. The Fourier transform of $x$
    is the polynomial $\hat{x}(\omega) = 9/2 - 9\omega -(1/2) \omega^2 + \omega^3$ and the Fourier intensity function is
    $$A(\omega) = |\hat{x}(\omega)|^2
    = (9/2) \omega^{-3} -(45/4) \omega^{-2} -(73/2) \omega^{-1} - (73/2) \omega - (45/4)
    \omega^2 + (9/2) \omega^3.$$
    By normalizing so that the first coordinate is positive real, we can eliminate the scaling ambiguity. The reflected vector $\dot{x} = (9/2, 1,-1/2,-9)$
    has Fourier transform $$\hat{\dot{x}} = 9/2 +\omega -(1/2) \omega^2 -9 \omega^3$$ which is the complex conjugate of $\hat{x}(\omega)$ since $\hat{\omega} =
    \omega^{-1}$ for $\omega \in S^1$.
  The real vectors 
  $$\begin{array}{ccl}
    x_2  & = & (9,-9/2,-1,1/2)\\
    x_3 &=& (3/2,-7,13/2,3)\\
    x_4 & = & (3/2,1,-19/2,3)
  \end{array}$$
      all have the same Fourier intensity function as $x$.
      and are unrelated by a trivial ambiguity. In total there are $2^4 =8$ vectors
      with positive first coordinate and Fourier intensity function
            $$(9/2) \omega^{-3} -(45/4) \omega^{-2} -(73/2) \omega^{-1} - (73/2) \omega - (45/4)
      \omega^2 + (9/2) \omega^3.$$
      They are $\{x, \dot{x}, x_1, \dot{x_1}, x_2, \dot{x_2}, x_3, \dot{x_3}\}$.

  Similarly, the real and complex vectors
  $$\begin{array}{ccl}
    x_1 & = & (9/2,9, 1/2,1)\\
    x_2 & = & (3/2,3 +4i, 3/2+ 8i, 3)\\
    x_3 & = & (3/2, 3 - 4i,3/2 - 8i, 3)\\
    x_4 & = &  (9, 9/2, 1, 1/2)
  \end{array}
  $$
    have the same Fourier intensity function 
    $$(9/2) \omega^{-3} + (45/4) \omega^{-2} + (91/2) \omega^{-1} + (205/2) + (91/2) \omega
    + (45/4) \omega^2 + (9/2) \omega^3$$
    and are unrelated by a trivial ambiguity.
  \end{Example}

  The discrete ambiguities of the phase retrieval problem can be understand
  algebraically as follows: If $x \in \C^{N+1}$ has full support
  then by the fundamental theorem of algebra we can factor
  the polynomial $\hat{x}(\omega)$ as $\hat{x}(\omega) = x_N(\omega - \beta_1)\ldots (\omega -\beta_N)$ for some non-zero complex numbers $\beta_1, \ldots , \beta_N$. (Note that the $\beta_i$ are all non-zero because $\displaystyle{x_0 =
    (-1)^N x_N \prod_{i=0}^N \beta_i}$ and $x_0, x_N$ must both be non-zero
  for $x$ to have full support.) Using the fact that $\overline{\omega} = \omega^{-1}$ on $S^1$ the Fourier intensity function factors as
  $$A(\omega) = \overline{x_0} x_N \omega^{-N} (\omega - \beta_1) (\omega - {{1\over {\overline{\beta_1}}})
      \ldots (\omega - \beta_N)(\omega -
             {1\over{\overline{\beta_N}}}}).$$
  In particular, if for each $k =1, \ldots , N$  we choose
  $\gamma_k \in \{\beta_k,
  {1\over{\overline{\beta_k}}}\}$
  then there is a constant  $b \in \C$ (whose modulus is unique)
  such that 
   $|b \prod_{i=1}^N(\omega - \gamma_i)|^2 = A(\omega)$. Reading off
  the coefficients of the polynomial
$b \prod_{i=1}^N(\omega - \gamma_i)$
  gives a new vector
           $x'$ such that $|\hat{x'}(\omega)|^2 = A(\omega)$. If we take
           $\gamma_k = {1\over{{\overline{\beta_k}}}}$ for all $k$
           the new vector is the vector $\dot{x}$ obtained by conjugating and reflecting
           the vector $x$. 
  
  By contrast, the 2D and higher phase retrieval problem is known to
  have a solution for almost all signals  \cite{BrSo:79, 
    KEO:17}. 
  Precisely almost all discrete functions  $f \colon
  \Z_N^2 \to \C$ are uniquely 
  determined modulo trivial ambiguities by the Fourier intensity
  function $|f(\omega, \eta)|^2$. The reason for the difference is
  that for generic (in the sense of complex algebraic geometry) $f$ the Fourier polynomial $f(\omega, \eta)$ is
  irreducible, while in one variable $f(\omega)$ always factors
into  distinct linear factors \cite{HaMc:82}.

Analyzing the possible signals with the same power spectrum naturally
arises in systems theory and digital signal processing. The method of
spectral factorization produces a signal with minimum phase; ie the
solution where $|\gamma_k|$ is minimized.  A similar approach is used
in filter design for more general systems associated to rational
functions, where only the magnitude response is determined
\cite[Section 5.6]{OpSc:10}.

However, without other information there is no reason for the minimum
phase signal to equal the desired signal. A goal of 
this paper is to understand how the possible factorizations vary with
the signal. This question is also related to questions in convex algebraic geometry involving sums of squares. In our case we are determining how the different ways to express the non-negative polynomial $A(\omega)  = |\hat{x}(\omega)|^2$
as a Hermitian
square $|p(\omega)|^2$ vary with $x$.
In the language of \cite{CPSV:17} we are studying how the rank one extreme points of the Gram spectrahedron of the non-negative polynomial
$|\hat{x}(\omega)|^2$ vary as $x$ varies.

To do this  we consider the geometry associated to the ambiguities
of the one-dimensional Fourier phase retrieval problem for vectors in
$\C^{N+1}$.  Our first results (Theorem \ref{thm.coverambig},
\ref{thm.coveruniqueness}) state that the
space of signals has a finite covering (which we call the root
covering) where any two signals in the covering with the same
Fourier intensity differ by a {\em trivial covering ambiguity}.
In other
words, we prove that phase retrieval is possible on the root cover.

Next we use these results to study how the non-trivial ambiguities of a signal vector vary as the
signal varies. To do this we describe (Theorem \ref{thm.incidence})
the {\em incidence variety} $I$
consisting of pairs $(x,x')$ with same Fourier intensity function
modulo global phase. We show that $I$ consists of $N+1$ connected irreducible
components, $I_0, \ldots , I_N$ and that the component $I_k$ is a
finite covering of degree ${N \choose k}$ of the space signals modulo
global phase.

Theorem \ref{thm.convolution} gives a geometric refinement of an earlier result of
Beinert and Plonka \cite[Theorem 2.3]{BePl:15}. Our result states
that the connected irreducible component $I_k$ of the incidence variety $I$
corresponds to pairs $(x,x')$ where
$x = x_1 \star x_2$, $x' = x_1 \star \dot{x}_2$ for some vectors
$x_1 \in \C^{k+1}, x_2 \in \C^{N-k+1}$. Here $x_1 \star x_2$ refers to the
convolution. (See the notation section for the definition
of the convolution.) 

As a consequence, if $k \neq 0, N$, then for a generic pair
$(x,x') \in I_k$, $x'$ is not obtained from $x$ by a trivial
ambiguity. We also prove that if $(x,x') \in I_k$ then $(x, \dot{x}')\in I_{N-k}$ where $\dot{x'}$ is obtained from $x'$ by reflection and
conjugation.

As an application we give (Theorem \ref{Thm:GenericUniqueness}) a criterion for a real subvariety $W$ of the
space of signals to admit generic phase retrieval.
Precisely we prove that if there exists a single signal $w_0 \in W$
with the property that any $w_0' \in W$ with the same Fourier
intensity function is obtained from $w_0$ by a trivial ambiguity
then the generic $w \in W$ has the same property. In other words, the
condition that a signal $w$ lies in the subvariety $W$ enforces uniqueness of
generic phase
retrieval provided there exists a single signal in $W$ with this property.
Examples of interesting $W$ include subvarieties of signals with a fixed
entry or sparse signals \cite{BePl:15, BePl:18}. This result
for tuples of signals was our original motivation 
for writing this paper. It plays a crucial role
in the author's work with Bendory and Eldar \cite{BEE:18}
proving that  a pair of signals can be recovered from their blind phaseless short-time Fourier transform measurements using $\sim 10N$ measurements where $N$ is the
length of the signal.

\subsection{Notation.} To slightly simplify our notation we
assume that our signals are vectors $x  \in \C^{N+1}$ as opposed
to vectors in $\C^N$ which is often used in the literature \cite{BePl:15, BePl:18}.
A vector $x \in \C^{N+1}$ can also be thought of as a function $\Z \to \C$ with support
in the interval $[0,N]$. We use the notation $x[n]$ to refer to the value
of this function at the integer $n$; i.e., if $x = (x_0, x_1, \ldots , x_N)$
then $x[n] = x_n$ for $n \in [0,N]$ and zero otherwise.

If $x \in \C^{N+1}$ then the {\em reflected vector} $x'$ is defined by the
formula $x'[n] = x[N+1-n]$ where the indices are taken modulo $N+1$;
i.e. $(x_0, x_1, \ldots , x_N)' = (x_0, x_N, x_{N-1}, \ldots x_1)$.

All signals $x \in \C^{N+1}$
are assumed to have full support. This means that $x[0], x[N]$
are both assumed to be non-zero. The set of such signals is parametrized by the
  complex variety $ \C^\times \times \C^{N-1} \times \C^\times$ which we view as
  a real variety of dimension $2N+2$. Because we work in $\C^{N+1}$ the Fourier intensity function $|\sum_{n=0}^N x[n]\omega^n|^2$ is a non-negative real trigonometric polynomial
  of degree $2N$ which is uniquely determined by its value at $2N+1$ distinct
  points on the unit circle.

  If $x_1 \in \C^{k+1}$ and $x_2 \in \C^{N-k+1}$ then the convolution
  $x_1 \star x_2$ is the vector in $\C^{N+1}$
  defined by the formula
  $$(x_1 \star x_2)[n] = \sum_{\ell=0}^k x_1[\ell] \overline {x_2[n-\ell]}.$$
  \subsection{Acknowledgments.} The results of this paper were inspired by the papers of Beinert and Plonka \cite{BePl:15, BePl:18} on the topic of ambiguities
  in Fourier phase retrieval. The author is grateful to Tamir Bendory for useful discussions as well as to the referees  for a careful reading and many helpful comments.
  
  \section{Background on algebraic geometry}
  \subsection{Real algebraic sets}
  A real algebraic set is a subset $X= V(f_1, \ldots ,f_r) \subset \R^M$
  defined by the simultaneous vanishing of polynomial equations
  $f_1, \ldots , f_r \in \R[x_1,\ldots, x_M]$.
  Note that any real algebraic set is defined by the single
  polynomial $F = f_1^2 +\ldots + f_r^2$.
Given an algebraic set $X = V(f_1, \ldots , f_m)$ we define the
Zariski topology on $X$ by declaring closed sets to be the
intersections of $X$ with other algebraic subsets of
$\R^M$. An algebraic set is {\em irreducible} if it is not the union
of proper algebraic subsets. An irreducible algebraic set is
called a {\em real algebraic variety}. Every algebraic set has a decomposition
into a finite union of irreducible algebraic subsets \cite[Theorem 2.8.3]{BCR:98}.

An algebraic subset of $ X \subset \R^M$ is irreducible if and only if
the ideal $I(X) \subset \R[x_1, \ldots , x_M]$ of polynomials
vanishing on $X$ is prime.  More generally we declare an arbitrary
subset of $X \subset \R^M$ to be irreducible if its closure in the
Zariski topology is irreducible. This is equivalent to the statement
that $I(X)$ is a prime ideal \cite[Theorem 2.8.3]{BCR:98}.

Note that in real algebraic geometry irreducible algebraic sets need not be
connected in the classical topology. For example the real variety defined by the equation $y^2 - x^3 +x$ consists of two connected components.

\subsection{Semi-algebraic sets and their maps}
In real algebraic geometry it is also natural to consider subsets
of $\R^M$ defined by inequalities of polynomials. A {\em semi-algebraic} subset of $\R^M$ is a finite union of subsets of the form:
\[ \{x \in \R^M; P(x) = 0\; \text{ and }\; \left(Q_1(x)>0, \ldots , Q_{\ell}(x) > 0\right)\} \]
Note that if $f \in R[x_1, \ldots , x_M]$ the set $f(x) \geq 0$
is semi-algebraic since it is the union of the set $f(x) =0$ with the set $f(x) > 0$.

The reason for considering semi-algebraic sets is that the image of an
algebraic set under a real algebraic map need not be real
algebraic. For a simple example consider the algebraic map $\R \to
\R,
x \mapsto x^2$. This map is algebraic but its image is the
semi-algebraic set $\{x\geq 0\} \subset \R$. A basic result in real
algebraic geometry states that the image of a semi-algebraic set under
a polynomial map is semi-algebraic, \cite[Proposition 2.2.7]{BCR:98}.

A map $f \colon X \subset \R^N \to Y \subset \R^M$ of semi-algebraic
sets is {\em semi-algebraic} if the graph $\Gamma_f = \{(x,f(x))\}$ is
a semi-algebraic subset of $\R^N \times \R^M$.  For example the map
$\R_{\geq 0} \to \R_{\geq 0}, x \mapsto \sqrt{x}$ is semi-algebraic
since the graph $\{(x,\sqrt{x})| x\geq 0\}$ is the semi-algebraic
subset of $\R^2$ defined by the equation $x=y^2$ and inequality $x
\geq 0$.  Again the image of a semi-algebraic set under a
semi-algebraic map is semi-algebraic.

\subsection{Dimension of a semi-algebraic sets}
A result in real algebraic geometry
\cite[Theorem 2.3.6]{BCR:98} states that any real semi-algebraic
subset of $\R^n$ admits a semi-algebraic homeomorphism\footnote{A semi-algebraic
  map $f \colon A \to B$ is a semi-algebraic homeomorphism if $f$ is bijective
  and $f^{-1}$ is semi-algebraic.}
to a finite disjoint union of 
hypercubes. Thus we
can define the dimension of a semi-algebraic set $X$
to be the maximal dimension of a hypercube
in this decomposition. This can be shown to be equal to
the Krull dimension of the Zariski closure of $X$ in $\R^M$ \cite[Corollary 2.8.9]{BCR:98}.
As a consequence we obtain the important fact that if $Y$ is a semi-algebraic
subset of an algebraic set $X$ with $\dim Y  < \dim X$ then $Y$ is a contained
in a proper algebraic subset of $X$.

\subsection{Finite coverings of semi-algebraic sets}
Following \cite{Edm:76} we say that a map of $f \colon X \to Y$ of
locally connected, connected Hausdorff topological spaces is {\em a
  finite or ramified cover} if it is open and closed and for all
$y \in Y$, $f^{-1}(y)$
is a finite non-empty set. Define the degree of $f$ to be $\sup \{
|f^{-1}(y)|, y \in Y\}$ with the convention that that $\deg f =
\infty$ if the supremum does not exist.  A result in point-set
topology \cite[Theorem I.10.2.1]{Bou:98} states that these conditions
are equivalent to the map $f$ being {\em proper}\footnote{A map $f
  \colon X \to Y$ of topological spaces is {\em proper} if for any
  topological space $Z$ the induced map $f \colon X \times Z \to Y
  \times Z$ is closed. This is analogous to the notion of universally
  closed in algebraic geometry. When $X,Y$ are locally compact this is
  equivalent to the more familiar notion that the inverse image of any
  compact set is compact, \cite[Proposition I.3.7]{Bou:98}} with
finite fibers.

In this paper all examples of finite coverings come from group actions.
If $X$ is a 
connected Hausdorff topological space and $G$ is a finite
group acting discretely on $X$ then set of orbits $X/G$ is also a
Hausdorff topological space and the orbit map $f \colon X \to X/G$ is a
finite covering. This follows from a result in general topology
\cite[Proposition III.4.2.2]{Bou:98} that states if
$G$ is compact (for example finite) then $X/G$ is Hausdorff and the orbit
map $X \to X/G$ is proper.

If $G$ acts almost freely, meaning that the set of points with trivial stabilizer is dense, then the degree of $f$
is $|G|$ because the fibers are orbits and the assumption implies
a dense set of orbits
has cardinality equal to $|G|$. 
A key fact about finite coverings is the following:

\begin{Proposition}
  Let $X\subset \R^N$, $Y \subset \R^M$ be semi-algebraic sets and let
  $f \colon X \to Y$ be a semi-algebraic map which is a finite covering.
  Then $\dim X = \dim Y$.
\end{Proposition}
\begin{proof}
  By the semi-algebraic triviality theorem \cite[Theorem
    9.3.2]{BCR:98}, $Y$ can be partitioned into a finite number of
  semi-algebraic sets $Y_1, \ldots, Y_r$ such that $f^{-1}(Y_\ell)$ is
  homeomorphic to $F_\ell \times Y_\ell$, where $F_\ell$ is the fiber
  over a point of $Y_\ell$.  Since $f$ is a finite cover, $F_\ell$ is
  a finite set and we conclude that $\dim f^{-1}(Y_\ell) = \dim
  Y_\ell$ since two homeomorphic semi-algebraic sets have the same
  dimension \cite[Theorem 2.8.8]{BCR:98}.  This also gives a partition
  of $X$ into a finite number of semi-algebraic sets of the same
  dimensions. Since the partition is finite we necessarily have that
  $\dim Y = \max_\ell \dim Y_\ell$ and likewise $\dim X = \max_\ell
  \dim f^{-1}(Y_\ell)$. Therefore $\dim X = \dim Y$.
  
\end{proof}

\subsection{Finite coverings and quotients by finite groups in complex algebraic geometry}
We briefly consider finite covers of complex algebraic varieties. A complex
algebraic subset 
$X \subset \C^M$ is the subset defined by the simultaneous vanishing of
polynomials $f_1, \ldots , f_r \in \C[x_1, \ldots, x_M]$. As in the real case we define the Zariski topology by declaring algebraic sets to be closed. An algebraic subset which is irreducible is called a variety. Unlike the real case,
any complex algebraic variety is connected \cite[Proposition 3.1.1]{BCR:98}.

If $X$ is a complex algebraic variety then the ring $\C[x_1, \ldots ,
  x_M]/I(X)$ is called the coordinate ring of $X$ where $I(X)$ is the
ideal of functions vanishing on $X$. We denote this ring by $\C[X]$.
The ring $\C[X]$ is the ring of polynomial functions on $X$. Because
$X$ is irreducible, $I(X)$ is a prime ideal so $\C[X]$ is an integral
domain. We denote its field of fractions by $\C(X)$.

Any polynomial map of complex varieties $f \colon X \subset \C^M \to Y
\subset \C^M$ is induced by a ring homomorphism $f^\sharp\colon \C[Y] \to
\C[X]$. It is defined by the formula $f^\sharp(h)(x) = h(f(x))$
where $h \in \C[Y]$. For more details see \cite[Section I.3]{Har:77}.

We say that a map of complex varieties is a {\em finite algebraic
 cover} if the map $f^\sharp$ is injective and $\C[X]$ is
finitely generated as a $\C[Y]$ module. The degree of $f$ is the
degree of the necessarily finite field extension $[\C(X):C(Y)]$. Any
finite algebraic cover $f \colon X \to Y$ of degree $d$ is also a
finite cover in the sense of topology where $X,Y$ are given the
subspace topologies induced by $\C^M$ and $\C^N$ respectively.
To see this we first note that a finite algebraic cover
has finite fibers 
\cite[Exercise 4.1]{Har:77} and is also projective. Hence,
if $f \colon X \subset \C^M \to Y \subset \C^N$ is a finite algebraic map
of varieties, then $X$ can be identified with a closed algebraic subset
of $\P^s \times Y$ for some $s\geq0$.
Since complex projective space can be embedded as a closed and bounded subset
of Euclidean space it is compact \cite[Proposition 3.4.11]{BCR:98}.
Thus the projection $\P^s \times Y \to Y$
is proper as a map of topological spaces. The map $f \colon X \to Y$
is the composition of a closed immersion with a proper map which implies
that it is also proper.

A finite group $G$ acts algebraically on a variety $X$ if for
each $g \in G$ the automorphism $X \to X$, $x \mapsto gx$ is a polynomial
map. In particular, the group $G$ acts on the coordinate ring $\C[X]$.
A fundamental result in invariant theory \cite{Fog:69} states that
the invariant subring $\C[X]^G:= \{h \in \C[X]| g\;h = h\;\; \forall g \in G\}$
is a finitely generated algebra, and that $\C[X]$ is a finitely generated
$\C[X]^G$ module. This means that there is a complex variety $Y$
whose coordinate ring is $\C[X]^G$ and the map $X \to Y$ is a finite cover.
In addition, if we view $Y$ as a subset of $\C^N$ then it can be identified
with the set of orbits  $X/G$.
As in topology, the algebraic degree of the finite cover $X \to X/G$ equals to
the maximal size of an orbit. See \cite[pp 124--126]{Har:95} for more details.
A deep
result on actions of algebraic groups states that the set of points
whose orbits have maximal size is Zariski open.

\section{Phase retrieval on the root covering}
To understand the non-trivial ambiguities we pass to an auxiliary
variety which we call the root covering. It parametrizes all orderings
of the roots of the Fourier polynomials of signals in $\C^\times \times
\C^{N-1} \times \C^\times$.  The root covering has a bigger group of
trivial ambiguities and we demonstrate that every vector in the root
covering is determined modulo trivial ambiguities from the Fourier
intensity function of the corresponding signal.

\subsection{The group of trivial ambiguities of the space of signals}
We begin by identifying a group of trivial ambiguities
acting on $\C^\times \times \C^N \
  \times \C^\times$ which preserves the Fourier intensity function.
   There is a natural free action
  of the circle group $S^1$ on $\C^\times \times \C^{N-1} \times \C^\times$
  where $e^{\iota \theta}$ acts on a vector $x$ by scalar
  multiplication. This action of of $S^1$ clearly preserves the
  Fourier intensity function $|\hx(\omega)|^2$.
    There is also an action of the group $\mu_2=\{\pm 1\}$ where the non-trivial
  element $(-1) \in \mu_2$ takes $x$ to $\dot{x}$ where $\dot{x}$ is
  obtained from $x$ by reflection and conjugation. The action of
  $\mu_2$ is not free since it fixes vectors $x\in \C^{N+1}$ with the property that
  where $x[k] = \overline{x[N-k]}$. However, it also preserves the Fourier
  intensity function since $\hat{\dot{x}} = \overline{\hat{x}}$.

  The group generated by $S^1$ and the conjugation reflection involution
  is the orthogonal group\footnote{The reason we take the semi-direct product rather than the
    product is the actions of $S^1$ and $\Z_2$ do not commute. The
    semi-direct product consists of pairs $(\lambda, \pm 1)$ but the
    multiplication is non-commutative.  Precisely, that $(\lambda,
    -1)(\mu, 1) = (\lambda \overline{\mu}, -1)$ while
    $(\mu,1)(\lambda,-1) = (\lambda \mu, -1)$.} $O(2) = S^1 \ltimes \mu_2$.
  We refer to this group as the group of {\em trivial ambiguities} of
  the phase retrieval problem. In classical Fourier phase retrieval
  (cf. \cite[Proposition 2.1]{BePl:15}) shifts are also considered
  to be trivial ambiguities. However, we eliminate the shift ambiguity from
  the outset by assuming our signals have fixed support $[0,N]$.

  The basic difficulty in phase retrieval is that the map
  $$(\C^\times  \times (\C^{N-1}) \times \C^\times)/(S^1 \ltimes \mu_2) \to \R^{2N+1}_{\geq 0}$$
  $x \mapsto |\hx(\omega)|^2$ is not injective.
  Indeed the generic fiber has $2^{N-1}$ points. The elements of the fiber
  are referred to as {\em non-trivial ambiguities}. In this paper we study
  how the non-trivial ambiguities vary with the signal.

\subsection{The root covering}

If $x \in \C^{N+1}$ with $x[N] \neq 0$, then Fourier transform 
$\hat{x}(\omega) = \sum_{n=0}^Nx[n] \omega^n$ is a polynomial
of degree $N$ on the unit circle. By the
fundamental theorem of algebra we can factor $\hat{x}(\omega) =
a_0(\omega - \beta_1) (\omega -\beta_2) \ldots (\omega -\beta_N)$.  If
we assume that $x$ has full support, then $x[0] = (-1)^Na_0
\beta_1\ldots \beta_N \neq 0$, so none of the roots of
$\hat{x}(\omega)$ are 0.

We denote $\C^\times \times (\C^\times)^N$ parametrizing tuples $(a_0, \beta_1,
\ldots , \beta_N)$ as the {\em root covering} of the space of
signals $\C^\times \times \C^{N-1} \times
\C^\times$.  The reason for this terminology is that we show in Proposition \ref{prop.rootcover}
that the  map $\Phi \colon \C^\times \times (\C^\times)^N \to \C^\times \times
\C^{N-1} \times \C^\times$ defined by the formula
\begin{equation}
(a_0,\beta_1, \ldots , \beta_N) \mapsto a_0\left(e_N(-\beta_1, \ldots ,
  -\beta_N),e_{N-1}(-\beta_1, \ldots , -\beta_N), \ldots
  e_1(-\beta_1, \ldots , -\beta_N),1\right)
  \end{equation}
where $e_n(-\beta_1, \ldots , -\beta_N)$ indicates the $n$-th elementary symmetric polynomial in $(-\beta_1, \ldots , -\beta_N)$ is a finite algebraic covering.

 By construction, the map $\Phi$ associates to the $(N+1)$-tuple $(a_0, \beta_1, \ldots ,
    \beta_N)$ a vector $x= (x_0, x_1, \ldots , x_N)$ whose Fourier
    transform factors as $$\hx(\omega)  = a_0(\omega - \beta_1)(\omega - \beta_2)
    \ldots (\omega - \beta_N).$$

    Note that the map $\Phi$ is multiple-to-one since any permutation of
    $(\beta_1, \beta_2, \ldots , \beta_N)$ produces the same vector.

    \begin{Proposition} \label{prop.rootcover}
      The map $\Phi$ is a finite algebraic covering of degree $N!$.
    \end{Proposition}

    \begin{Example} \label{ex.rootcover} Consider the vector $x=(9/2,9,1,1/2,1)$. Its Fourier
      transform is $\hat{x}(\omega) = (\omega - 3\iota)(\omega + 3\iota)(\omega + 1/2)$. The inverse image of $x$ in the root cover consists of the 6 vectors
      $$\begin{array}{ccc}
        \tilde{x}_1 & = & (1, 3\iota, -3\iota, -1/2)\\
        \tilde{x}_2 & = & (1,  3\iota, -1/2, -3\iota)\\
        \tilde{x}_3 & = & (1,-3\iota, 3\iota, -1/2)\\
        \tilde{x}_4 & = & (1, -3\iota, -1/2, 3\iota)\\
        \tilde{x}_5 & = & (1,-1/2, 3\iota, -3\iota)\\
        \tilde{x}_6  & = & (1,-1/2, -3\iota, 3\iota).
      \end{array}
      $$
Note that $N+1 =4$ in this example.
      \end{Example}
    \begin{proof}   

   The map $\Phi$ is the composition $\sigma \circ \pi$ where $\pi
   \colon \C^\times \times \C^N \to \C^\times \times \C^N$ is the map $$(a_0,
   \beta_1, \ldots, \beta_N) \mapsto (e_N(-\beta_1, \ldots ,
   -\beta_N), \ldots , e_1(-\beta_1, \ldots, -\beta_N),a_0)$$ and
   $\sigma(x_0, x_1, \ldots, x_N) = (x_Nx_0, x_Nx_1, \ldots, x_Nx_{N-1}, x_N).$
   The map $\sigma$ is an isomorphism of complex varieties with
   inverse given by $(x_0, x_1, \ldots , x_N) \mapsto (x_0x_N^{-1}, x_0x_N^{-1}
, \ldots , x_{N-1}x_N^{-1},x_N)$.

    The map $\pi$ a finite algebraic cover of complex
    varieties of degree $N!$ since it is the $S_N$ quotient map
    $\C^\times \times \C^N \to \C^\times \times \C^N$ where $S_N$ acts by permuting the last $N$ coordinates.
    This fact follows from the classical fundamental theorem of symmetric polynomials.
    It states that every symmetric polynomial can be uniquely expressed
    as a polynomial in the elementary symmetric functions
    and the polynomial ring is a free module over the ring of symmetric functions
    of degree $N!$. For a reference
    see \cite[Chapter 7, Theorem 3]{CLO:15}.

    Since $\sigma$ is an
    isomorphism this means we can identify $\C^\times \times
    \C^{N-1} \times \C^\times$ as the quotient of $(\C^\times)^{N+1}$ by the
    action of the symmetric group $S_N$, where $S_N$ acts by permuting
    the last $N$ factors. Since $S_N$ acts with generically trivial stabilizer 
     $\Phi$ is a finite algebraic covering of degree $N!$.
 \end{proof}
 If we view $\Phi$ as a map of real algebraic varieties then
 Proposition \ref{prop.rootcover} implies that $\Phi$ is a finite covering of degree $N!$ in the sense of real algebraic geometry.
 
 \subsection{The group of ambiguities of the root covering}
 We now consider the group of ambiguities of the root
 cover. Precisely we consider a group $G$ acting faithfully
 on $(\C^\times)^{N+1}$
 such that for all $\tilde{x} \in (\C^\times)^{N+1}$ with $\Phi(\tilde{x})
 = x$ and $g \in G$ with $\Phi(g \tilde{x}) = x'$ then
 $|\hx(\omega)|^2 = |\hx'(\omega)|^2$.

 \begin{Theorem}[The ambiguity group of the root cover] \label{thm.coverambig}
   The group $G = S^1 \ltimes \left((\mu_2)^N \ltimes S_N\right)$ is a group
   of  ambiguities for phase retrieval on $(\C^\times)^{N+1}$.
   \end{Theorem}
 
We refer to $G$ as the {\em root ambiguity group}.

 \begin{proof}[Proof of Theorem \ref{thm.coverambig}]
 We first describe the Fourier intensity preserving action of $G = S^1
 \times \left( (\mu_2)^N \ltimes S_N\right)$ on $(\C^\times)^{N+1}$.

 The
 action of $S^1$ is given as follows: If
 $\tilde{x} = (a_0, \beta_1, \ldots, \beta_N)$ then
  $\lambda \cdot \tilde{x} = (\lambda a_0, \beta_1, \ldots , \beta_N)$. The
 effect of the action of $S^1$ on $\Phi(\tilde{x})$ is to multiply each
 entry of $\Phi(\tilde{x})$ by the scalar $\lambda$. Since $\lambda
 \in S^1$ this does not change the Fourier intensity function.

 We now describe the action of $(\mu_2)^N \ltimes S_N$.
 The symmetric group $S_N$
 acts by permuting $\beta_1,\ldots , \beta_N$. Since the elementary
 symmetric polynomials are invariant under permutations of $\beta_1,
 \ldots , \beta_N$, if $\tau \in S_N$ then $\Phi(a_0, \beta_1,
 \ldots , \beta_N) = \Phi(a_0, \beta_{\tau(1)}, \ldots ,
 \beta_{\tau(N)})$, so $\Phi(\tau \cdot \tilde{x}) = \Phi(\tilde{x})$.

 The group $(\mu_2)^N$ is generated by elements $s_i = (1,\ldots ,
 1,-1,1,\ldots 1)$ where the $-1$ is in the $i$th position. The
 element $s_i$ acts on $\tilde{x} = (a_0, \beta_1, \ldots , \beta_N)$
 by 
 $$s_i\cdot \tilde{x} = (a_0|\beta|_i, \beta_1, \ldots , \beta_{i-1}, \overline{\beta_i}^{-1}, \beta_{i+1}, \ldots , \beta_N).$$

 The actions of $S_N$ and $\mu_2^N$ do not commute since
 $\tau s_i \tilde{x} = s_{\tau(i)} \tau \tilde{x}$. Thus we have an action
 of the semi-direct product of $\mu_2^N \ltimes S_N$ where $S_N$ acts on $\mu_2^N$ by permutation.
 (Note that the action of $\mu_2$ is only semi-algebraic because we need to multiply by $|a_i|$ in order to ensure that $s_i^2$ acts as the identity.)
 
 Let us verify that if
 $x' = \Phi(s_i \cdot \tilde{x})$ and $x = \Phi(\tilde{x})$ then $x'$
 and $x$ have the same Fourier intensity function.
 The Fourier transform of $x$ is $\hat{x}(\omega) = a_0  \prod_{i=1}^N (\omega - \beta_i)$. Thus
 $$|\hat{x}(\omega)|^2 = |a_0|^2 \prod_{i=1}^N (\omega - \beta_i)(\omega^{-1} -
 \overline{\beta}_i)$$
 while
 $$\hat{x'}(\omega) = a_0 \beta_i (\omega - \beta_1)\ldots (\omega -
 \beta_{i-1}) (\omega - \overline{\beta}_i^{-1})(\omega - \beta_{i+1})
 \ldots (\omega - \beta_N)$$
 so
 \begin{multline*} |\hat{x'}(\omega)|^2 =
 |a_0 \beta_i|^2 (\omega + \beta_1) (\omega^{-1} + \overline{\beta}_1)
 \ldots (\omega + \beta_{i-1})(\omega^{-1} + \overline{\beta}_{i-1})\\
 (\omega + \overline{\beta_i}^{-1})(\omega^{-1}  + \beta_i^{-1}) (\omega + \beta_{i+1})(\omega^{-1} + \overline{\beta}_{i+1}) \ldots (\omega + \beta_N)(\omega^{-1}
 + \overline{\beta}_N)
 \end{multline*}
 Since $$(\omega + \overline{\beta_i}^{-1})(\omega^{-1}  + \beta_i^{-1}) =
 {1\over{\beta_i \overline{\beta}_i}} (\omega^{-1} + \overline{\beta}_i)(\omega + \beta_i)$$
 we see that the two Fourier intensity functions are the same.
 Finally, note the actions of $S_N$ and $\mu_2^N$ do not commute since
 $(\tau s_i) \tilde{x} = (s_{\tau(i)} \tau) \tilde{x}$ which corresponds to an action
 of the semi-direct product $(\mu_2)^N \ltimes S_N$ where $S_N$ acts on $
 (\mu_2)^N$ by permutations.
 \end{proof}
 \begin{Remark}
   The characterization of \cite[Theorem 2.3]{BePl:15} shows that the action of the $G$ covers all trivial and non-trivial ambiguities of the phase retrieval. Thus Theorem \ref{thm.coverambig} is an algebraic representation of the first statement of \cite[Theorem 2.3]{BePl:15}.
   \end{Remark}
\begin{Example}
Consider the vector $\tilde{x}_1 =  (1, 3\iota, -3\iota, -1/2)$ of Example \ref{ex.rootcover}.
Its image under the map $\Phi$ is the vector $(9/2, 9,1/2, 1)$ considered in Example \ref{ex.ambig}.
 Its orbit
under $(\mu_2)^3 \ltimes S_3$ (which is the discrete part of the ambiguity group $G$ when $N = 3$) consists of 48 vectors. Let us see how various group elements act on
$\tilde{x}_1$. 

The element $g=((1,-1,1) , id)$ moves $\tilde{x_1}$ to the vector
$\tilde{x}_1' = (3, 3 \iota, -(1/3)\iota, -1/2).$ Observe that $\Phi(\tilde{x}_1) = (9/2, 9, 1/2,1)$
and that $\Phi(\tilde{x}_1') = (3/2,3+4\iota ,3/2+8\iota,3)$. In the notation of Example \ref{ex.ambig}
this is the vector $x_3$.

The element $h = ((1,1,1), (12))$ moves $\tilde{x}_1$ to the vector $\tilde{x}_3 = (1, -3\iota, 3\iota, -1/2)$. Here $\Phi(\tilde{x}_1) = \Phi(\tilde{x}_2) = (9/2, 9, 1/2,1)$. If
we apply $g$ to $\tilde{x}_3 = h \tilde{x}_1$ we obtain the vector $(3,-3\iota, (1/3)\iota,-1/2)$
whose image under $\Phi$ is the vector $\displaystyle{x_2 = (3/2,3+4\iota,3/2+8\iota,3)}$.

On the other hand
if we apply $h$ to the vector $\tilde{x}_1' = g\tilde{x}$ we obtain the vector
$(3, -(1/3)\iota, 3\iota, -1/2)$. The image of this vector under $\Phi$ is  the vector
$x_3$. 
\end{Example}
 \subsection{Phase retrieval on the root cover}
 Our next result shows that phase retrieval is possible on the root
 coverings modulo its larger group of ambiguities. In other words,
 every vector $\tilde{x} \in (\C^\times)^{N+1}$ can be recovered from the corresponding Fourier intensity function  up to the action of the group
 $G = S^1 \ltimes (\mu_2^N \ltimes S_N)$.
 
  \begin{Theorem}[Phase retrieval on the root cover] \label{thm.coveruniqueness}

   Every $\tilde{x}$ can be uniquely determined modulo the root ambiguity group
   $G$ from the Fourier intensity function of $\Phi(x)$.

   In other words the map
   $(\C^\times)^{N+1}/G \to \R^{2N+1}_{\geq 0}$ which sends the orbit
   of $\tilde{x}$ to the coefficients of the Fourier intensity function
   of $\Phi(x)$ is well-defined and injective.
 \end{Theorem}
\begin{proof}[Proof of Theorem \ref{thm.coveruniqueness}]

  Suppose that $x=\Phi(\tilde{x})$ and $x' =\Phi(\tilde{x'})$ have the same
  Fourier intensity function where
  $\tilde{x} = (a_0, \beta_1 , \ldots , \beta_N)$ and
  $\tilde{x'} = (a'_0, \beta'_1, \ldots , \beta'_N)$.
  We wish to show that $\tilde{x}$ can be obtained from $\tilde{x'}$
  by the action of the root ambiguity group $G$.
  
  Expanding out the Fourier intensity functions we
  have
  $$|\hat{x}(\omega)|^2 =
  \omega^{-N}(|a_0|^2 \prod_{i=1}^N \overline{\beta}_i)\prod_{i=1}^N(\omega -\beta_i)(
  \omega - \overline{\beta}_i^{-1})$$
  and
 $$ |\hat{x'}(\omega)|^2 =
  \omega^{-N}(|a'_0|^2 \prod_{i=1}^N \overline{\beta'}_i)\prod_{i=1}^N(\omega -\beta'_i)(
  \omega - \overline{\beta'}_i^{-1}).$$
  Since the polynomial ring in one-variable is a unique factorization domain we must have that
  $|a_0|^2 \prod \overline{\beta_i} = |a'_0|^2 \prod \overline{\beta'_i}$
  and an equality of sets
  $\{\beta_1, \overline{\beta}_1^{-1}, \ldots, \beta_N, \overline{\beta}_N^{-1}\}
  = \{\beta'_1, \overline{\beta'}_1^{-1}, \ldots , \beta'_N, \overline{\beta'}_N^{-1}\}$.
  Hence after reordering the $\beta_1, \ldots , \beta_N$, which corresponds to
  applying a permutation to $\tilde{x}$, we may assume that
  $\beta'_i \in \{\beta_i, \overline{\beta}_i^{-1}\}$. Let
  $S$ be a subset of $\{1, \ldots , N\}$ such that
  $\beta'_i = \overline{\beta_i}^{-1}$ if $i \in S$ and
  $\beta'_j = \beta_j$ if $j \in S^{c}$.
  (Note that $S$ is uniquely determined if and only if none of the $\beta_i$
  lies on the unit circle.)
  Let $s =\prod_{i \in S} s_i$. Then
   $s \tilde{x'} =
  (a'_0 \prod_{i \in S}\overline{\beta}^{-1}_i, \beta_1, \ldots , 
  \beta_N)$.
  Since $\Phi(s \tilde{x'})$ and $\Phi(\tilde{x})$ have the same Fourier intensity
  function, we conclude that $|a'_0 \prod_{i \in S} \overline{\beta}^{-1}_i| = |a_0|$. Hence there is a scalar $\lambda \in S^1$ such that
  $\lambda s \tilde{x'} = x$.
\end{proof}
  \begin{Remark}
    Theorem \ref{thm.coveruniqueness} is an algebraic characterization of the second statement of \cite[Theorem 2.3]{BePl:15}. That theorem would then imply that any two roots with the same Fourier intensity function can only differ by an action of the root ambiguity group.
    \end{Remark}
  
\subsection{The Fourier intensity map for signals modulo trivial ambiguities}
The space of signals modulo trivial ambiguities is the quotient
of the variety $\C^\times \times \C^{N-1} \times \C^\times$ by the group
$S^1 \ltimes \mu_2$. Since $\C^\times \times \C^{N-1} \times \C^\times$ is the quotient
of $\C^\times \times (\C^\times)^N$ by $S_N$ we can realize the quotient of $\C^\times \times
\C^{N-1} \times \C^\times$ by its group of ambiguities as a quotient of the root cover
$\C^\times \times (\C^\times)^N$.

\begin{Proposition}\label{prop.quotient}
  The space of signals modulo trivial ambiguities
  $(\C^\times \times \C^{N-1} \times \C^\times)/S^1 \ltimes \mu_2$ is homeomorphic to the quotient
  of $(\C^\times \times (\C^\times)^N)$ by a subgroup $H$ of the root ambiguity group  $G$
  of index $2^{N-1}$.
\end{Proposition}
\begin{proof}
  Let $H$ be the subgroup of $G= S^1 \ltimes (\mu_2^N \ltimes S_N)$
  consisting of triples $(\lambda, s, \tau)$ where $\lambda \in S^1$,
  $\tau \in S_N$ and $s = (1, \ldots, 1)$ or $s = (-1, \ldots ,
  -1)$. Since $(-1, \ldots , -1)$ and $(1,\ldots , 1)$ are invariant
  under the action of permutations, this subgroup is isomorphic to the
  semi-direct product $S^1 \ltimes ( \mu_2 \ \times S_N)$. Moreover,
  the group $S_N$ acts trivially on $S^1$ so this semi-direct product
  is the same as $S_N \times (S^1 \ltimes \mu_2)$.  In particular,
  $S_N$ is a normal subgroup. Taking the quotient by the action of
  $S_N$ produces $\C^\times \times \C^{N-1} \times \C^\times$ with a
  residual action of the quotient group $S^1 \ltimes \mu_2$.

  To complete the proof of Proposition \ref{prop.quotient} we need show
  that the involution of $\C^\times \times \C^{N-1} \times \C^\times$ coming
  from $(-1) \in \mu_2$ is the involution $x \mapsto \dot{x}$.
  This follows from the following lemma proved in \cite{BePl:18}. 
  \begin{Lemma}{\cite[Lemma 2.5]{BePl:18}} \label{lem.rootsofthedot}
    If $\beta_1, \ldots , \beta_N$ are the roots of $\hat{x}(\omega)$,
    then the roots of $\hat{\dot{x}}(\omega)$ are
    $\overline{\beta_1}^{-1}, \ldots , \overline{\beta_N}^{-1}$.
\end{Lemma}
\end{proof}  

\begin{Remark}
  One might hope that there is a larger group $G$ of ambiguities acting
  on $\C^\times \times \C^{N-1} \times \C^\times$ such that the Fourier intensity function is injective modulo this group. Such a group would necessarily
  be a quotient of the root ambiguity group $G$ by the symmetric group
  $S_N$. However, $H$ is not a normal subgroup of the full root ambiguity group $G$ so the quotient $G/H$ is not a group. There is  a map of quotients
  $$(\C^\times  \times \C^{N-1} \to \C^\times)/(S^1 \ltimes \mu_2) = (\C^\times\times (\C^\times)^N)/H
  \to (C^\times \times (\C^\times)^N)/G$$
  which is a  $G/H$ covering. This is a finite covering
  of connected, irreducible semi-algebraic sets
degree $|G/H| = 2^{N-1}$ corresponding to the $2^{N-1}$ vectors modulo trivial ambiguities with the same Fourier intensity function.

  Precisely, the Fourier intensity map
  $(\C^\times \times (\C^\times)^N)/H =(\C^\times \times \C^{N-1} \times \C^\times)/(S^1 \ltimes \mu_2) \to \R^{2N+1}_{\geq 0}$
  factors as
  $$\xymatrix{(\C^\times \times \C^{N-1} \times \C^\times)/(S^1 \ltimes \mu_2) \ar[dr] \ar[d] &\\
    (\C^\times \times (\C^\times)^N)/G \ar[r] & \R^N_{\geq 0}}$$
  where the bottom arrow is injective and the diagonal arrow is a finite covering
  of degree $2^{N-1}$.
\end{Remark}

 \section{The incidence variety of ambiguities}

    Let $X$ be the quotient of the space $\C^\times \times \C^{N-1} \times \C^\times$
    by the the free action of $S^1$.
    The semi-algebraic map $(a_0, a_1 \ldots , a_{N-1}, a_N) \to (|a_0|, {\overline{a}_0\over{|a_0|}} a_1, \ldots , {\overline{a}_0 \over{|a_0|}}a_N)$
    identifies $X$ with the 
    semi-algebraic variety
    $\R_{>0} \times \C^{N-1} \times \C^\times$. The space
    $X$ is the space of equivalence classes of signals modulo
    global phase. Since $S^1$ is a normal subgroup of $S^1 \ltimes \mu_2$
    there is an action of $\mu_2$ on $X$. If $x \in X$ is represented by
    $(a_0, a_1, \ldots , a_{N-1}, a_N)$ with $a_0 \in \R_{>0}$, then
    $(-1) \cdot x$ is represented by the vector $(a_0, \overline{a_N},
    \ldots , \overline{a}_1)$. 

Let $I \subset X \times X$ be the subset of pairs $(x,x')$ of
equivalence classes of signals such that $|\hat{x}(\omega)|^2 =
|\hat{x'}(\omega)|^2$. We call $I$ the Fourier intensity incidence
correspondence.  Since $I$ is defined by real algebraic equations we
say that $I$ is a real algebraic subset of the semi-algebraic set $X
\times X$.  The goal of this section is to describe the decomposition
of $I$ into irreducible components.

  \begin{Theorem} \label{thm.incidence}
(i) The real algebraic subset $I \subset X \times X$ decomposes into $N+1$ irreducible components  $I_0, \ldots , I_N$ each of which is connected.

 (ii) The projection $I_k \to X$
    is a finite cover of degree ${N \choose k}$.

 (iii) The total degree of the map
    $I \to X$ is $\displaystyle{\sum_{n=0}^N {N \choose k} = 2^N}$
    and $I_0$ and $I_N$ are both isomorphic to $X$ as semi-algebraic sets.

 (iv) There is an additional action of $\mu_2$ on $I$ given by $(x,x') \mapsto
    (x, \dot{x}')$. Under this action $I_k \mapsto I_{N-k}$.

(v) If $(x,x') \in I_k\smallsetminus (I_k \cap (I_0 \cup I_n))$ 
      then $x'$ is not obtained from $x$ by a trivial ambiguity.
     \end{Theorem}

    \begin{Remark} We denote the union $\bigcup_{k \neq 0, N} I_k$ by $I^0$. The generic point
    of $I^0$ is a pair of $S^1$-equivalence classes $(x,x')$ such that
    $|\hat{x}(\omega)|^2 = |\hat{x'}(\omega)|^2$ but $x'$ is not obtained from
    $x$ by a trivial ambiguity.
\end{Remark}

  \begin{Remark} For a generic vector $x \in \C^{N+1}$
    there are, modulo
    trivial ambiguities, $2^{N-1}$
    vectors $x'$ such that $|\hat{x}(\omega)|^2 = | \hat{x'}(\omega)|^2$. This follows from our result since the finite covering $(I/\Z_2) \to X$ has degree
    $2^N/2 = 2^{N-1}$ so the generic fiber has $2^{N-1}$ points.
    The $2^{N-1}$ points are partitioned into
    $\lceil (N+1)/2 \rceil$ components corresponding to the possible non-equivalent convolutions $x_1 \star \dot{x_2}$ with
    $x_1 \in \C^{k+1}, x_2 \in \C^{N-k+1}$. See Section
    \ref{sec.convolution} for further discussion.
  \end{Remark}

  \subsection{The incidence correspondence on the root covering}
  To prove Theorem \ref{thm.incidence} we again pass to the root covering.
  
  Let $\tX$ be the quotient of
  $\C^\times \times (\C^\times)^N$ by the free action of $S^1$ on $(\C^\times)^{N+1}$ given by
    $e^{\iota \theta}(a_0, \beta_1, \ldots , \beta_N)=
  (e^{\iota\theta} a_0, \beta_1 , \ldots , \beta_N)$.
  The semi-algebraic map
$$(a_0,  \beta_1, \ldots , \beta_N) \mapsto (|a_0|, \beta_1, \ldots , \beta_N)$$
  identifes $\tX$ with 
    $\R_{>0} \times (\C^\times)^{N}$.

    The map $\Phi$ is $S^1$-equivariant where $S^1$ acts on $(\C^\times)^{N+1}$
    and $\C^\times \times \C^{N-1} \times \C^\times$ as above. The  action
    of $S_N$ also commutes with the $S^1$ action. Hence there is an induced
    map $\tilde{\Phi} \colon \tilde{X} \to X$ which identifies $X$ as the quotient of $\tilde{X}$
    by $S_N$. 
    
    As a consequence of Proposition \ref{prop.rootcover}  we have
    \begin{Proposition}
      The map $\tilde{\Phi}$ is a finite algebraic covering of degree $N!$ \qed
      \end{Proposition} 
    Let  $\tI$ be the following subset of $\tilde{X} \times \tilde{X}$:
      $$\tI := \{(\tx=(a_0, \beta_1, \ldots , \beta_N), \tx'=
    (a'_0, \beta'_1, \ldots , \beta'_N)|  \; |\hat{x}(\omega)|^2  = |\hat{x'}(\omega)|^2
     \; \text{and} \; \forall n\; \beta'_n \in \{ \beta_n, \overline{\beta_n}^{-1}\}\}$$
    where $x = \tilde{\Phi}(\tx)$ and $x' = \tilde{\Phi}(\tx')$.
    We refer to $\tI$ as the {\em root incidence variety}.
Again $\tilde{I}$ is a real algebraic subset of the semi-algebraic set
$\tilde{X} \times \tilde{X}$.
         
    \begin{Proposition}
      The  incidence $\tilde{I}$ decomposes into $2^N$ irreducible components
      each isomorphic via a semi-algebraic isomorphism to $\tilde{X}$
      embedded as the diagonal in $\tilde{X} \times \tilde{X}$. In particular, each irreducible component is connected.
      
      \end{Proposition}
        
        \begin{proof}
          Let $\tilde{x} = (a_0, \beta_1, \ldots , \beta_N)$ and
          $\tilde{x'} = (a'_0, \beta'_1, \ldots , \beta'_N)$
          be vectors in $\tilde{X}$ and let $x = (a_0, a_1, \ldots , a_N)$
          and $x' = (a'_0, a'_1, \ldots , a'_N)$ be their images in $X$.
          By the proof of Theorem \ref{thm.coveruniqueness} we know that $|\hat{x}(\omega)|^2  =
          |\hat{x}(\omega)|^2$ if and only if after possibly reordering the
          $\beta_i$ there exists a subset
          $S \subset \{1, \ldots, N\}$ such that $\beta'_i = \overline{\beta_i}^{-1}$ for $i \in S$ and $\beta'_i = \beta_i$ if $i \in S^c$ and
          $\prod_{i=1}^N {\beta_i\over{\beta'_i}}  = (a_0/a'_0)^2$.

          Hence $\tilde{I}$ is the union of $2^N$ closed real algebraic subsets
          indexed by subsets of $\{1, \ldots , N\}$. Specifically if $S$ is a subset then
          we let 
          $$\tilde{I}_S = \{(a_0, \beta_1, \ldots , \beta_N), (a'_0, \beta'_1, \ldots , \beta'_N)| \beta'_i = \overline{\beta_i}^{-1} \;\; \text{for}\; i \in S, \beta'_j = \beta_j\;\;\text{for}\; j \notin S,
          \prod_{i=1}^N {\beta_i\over{\beta'_i}}  = (a_0/a'_0)^2\}.$$

          Each of the $\tilde{I}_S$ is connected and irreducible because there is a semi-algebraic isomorphism
          $\tilde{X} \to \tilde{I}_S$ and $\tilde{X}$ is connected and irreducible. The isomorphism is given by the formula
          $$(a_0, \beta_1, \ldots , \beta_N) \mapsto \left((a_0, \beta_1, \ldots , \beta_N), (a_0', \beta'_1, \ldots ,
          \beta'_N)\right)$$ where $\beta'_i = \overline{\beta_i}^{-1}$ if $i \in S$,  $\beta'_i = \beta_i$ if $i \notin S$ and
          $$ a'_0 = \sqrt{a_0 \left( \prod_{i=1}^N {\beta'_i \over{\beta_i}}\right)}.$$
          
\end{proof}
\begin{Remark}
        Note that the $\tilde{I}_S \cap \tilde{I}_S'$ can be identified with the real subvariety
        of $\tilde{X}$ consisting of tuples $\tilde{x} = (a_0, \beta_1, \ldots , \beta_N)$ where $|\beta_i| = 1$ for $i\in
       (S \cup S') \smallsetminus
        (S \cap S')$. Hence
        $\cap_{S} \tilde{I}_S$ can be identified with $\R_{>0} \times (S^1)^N$, corresponding to vectors all of whose Fourier roots lie on the unit circle.
        \end{Remark}
          
  \subsection{Proof of Theorem \ref{thm.incidence}}
  To prove the theorem we need to understand the images in $I$
  of the irreducible components $\tilde{I}_S$ of $\tilde{I}$.

  \begin{Lemma}
    The image of $\tilde{I}_S$ equals the image of $\tilde{I}_{S'}$ if and only
    $|S|=|S'|$.
  \end{Lemma}
  \begin{proof}
    If $|S| = |S'|$ then there is a permutation $\tau \in S_N$ such
    that $\tau(S) = S'$. Under the diagonal action of $S_N$
    on $\tilde{I}$ given by
    $$\tau\left((a_0, \beta_1, \ldots, \beta_N), (a_0', \beta'_1, \ldots , \beta'_N)\right)=
    \left((a_0, \beta_{\tau(1)}, \ldots, \beta_{\tau(N)}), (a_0',
    \beta'_{\tau(1)}, \ldots , \beta'_{\tau(N)})\right)$$
    $\tilde{I}_S$ is mapped to $\tilde{I}_{S'}$. Since the map
    $\tilde{I} \to I$ obtained by restricting $\tilde{\Phi} \times \tilde{\Phi}$ to $\tilde{I}$
    is $S_N$ invariant, it follows that
    $\tilde{I}_S$ and $\tilde{I}_{S'}$ have the same image in $I$.

    Conversely suppose that $|S| \neq |S'|$. Without loss of generality we may assume that $|S| < |S'|$. Also we can find a permutation $\tau$ such that
    $\tau(S)$ is a proper subset of $\tau(S')$. Applying another permutation allows us to assume that $S =\{1, \ldots , k\}$ and $S'= \{1, \ldots , l\}$
    with $l > k$.
    
    If $\beta_1, \ldots, \beta_N$ are chosen to be distinct and none of them lie on the unit circle (for example we can take the $\beta_i$ to be positive real numbers more than 1) then the image of the pair
    $$(\tilde{x}, \tilde{x'}) = \left( (a_0,\beta_1, \ldots ,\beta_N), 
    (a_0', \overline{\beta_1}^{-1}, \ldots , \overline{\beta_l}^{-1}, \beta_{l+1}, \ldots , \beta_N) \right) \in \tilde{I}_{S'}$$ is not in the image of $\tilde{I}_S$.
    Likewise,
     $$(\tilde{x}, \tilde{x'}) = \left( (a_0,\beta_1, \ldots ,\beta_N), 
    (a_0', \overline{\beta_1}^{-1}, \ldots , \overline{\beta_k}^{-1}, \beta_{k+1}, \ldots , \beta_N) \right) \in \tilde{I}_{S}$$ is not in the image of $\tilde{I}_{S'}$.

  \end{proof}
  \begin{proof}[Proof Theorem \ref{thm.incidence}]
(i)  Since each $\tilde{I}_S$ is irreducible and connected, their images are irreducible so
  $I$ consists of $N+1$ irreducible and connected components $I_0, \ldots , I_N$
  where $I_k$ is the image of $\tilde{I}_S$ for any subset $S \subset \{1, \ldots, N\}$
  such that $|S| = k$. (This includes the empty set.)

  \medskip

  (ii,iii) We now compute the degree of the projection $I_k \to X$. We know that
  if $S$ is any subset with $|S| = k$ then the map
  $\tilde{I}_S \to I_k \to X$ has degree $N!$ since $I_S$ is homeomorphic to $\tilde{X}$.
  Two general elements of $\tilde{I}_S$ have the same image in $I_k$ if and only
  if there is a permutation $\tau \in S_N$ such that $\tau(S) = S$
  and $\tau(S^c) = S^c$. Hence $I_k$ may be identified with the quotient
  of $\tilde{I}_S$ by a subgroup of $S_N$ isomorphic to $S_k \times S_{N-k}$.
  Hence the degree of of the map $\tilde{I}_S \to I_k$ is $k!(N-k)!$. Since
  the degree of a finite map is multiplicative it follows that
  the degree of the map $I_k \to X$ equals to ${N!\over{k! (N-k)!}} = {N \choose k}$.

  \medskip

  (iv)
  The involution (order two automorphism) of $\tilde{I}$
   given by
   $$\left((a_0, \beta_1, \ldots , \beta_N),(a'_0, \beta'_1, \ldots ,
   \beta'_N)\right) \mapsto
   \left((a_0, \beta_1, \ldots , \beta_N),(a'_0\overline{\beta_1}'
   \ldots \overline{\beta_N}', (\overline{\beta'}_1)^{-1}, \ldots,
   (\overline{\beta'}_N)^{-1}\right))$$
   takes ${\tilde I}_S \to {\tilde I}_{S^c}$.
   Given $(\tilde{x}, \tilde{x}') \in \tilde{I}$, let
   $(\tilde{x}, \tilde{\dot{x}}')$ be its image under the involution.
   If $(x,x')$ is the image in $I$ of $(\tilde{x}, \tilde{x'})$ then
   the image in $I$ of $(\tilde{x},\tilde{\dot{x}}')$ is $(x,\dot{x}')$ where
   $\dot{x}'$ is obtained by conjugation and reflection.
   If $|S| = k$ then $|S^c| = |N-k|$, so we see that $I_k \mapsto I_{N-k}$
   under the involution $(x,x') \mapsto (x,\dot{x}')$.

   \medskip
   
(v)  Given $x \in X$, let $\beta_1, \ldots , \beta_N$ be the roots of the Fourier polynomial $\hat{x}(\omega)$.  For generic $x$ none of the roots
  $\beta_1, \ldots ,\beta_N$ lie on the unit circle. If $(x,x') \in I_k$ and $\beta'_1, \ldots , \beta'_N$
  are the roots of $\hat{x}'(\omega)$ then there is a subset $S \subset \{1, \ldots , N\}$ such that $\beta'_i = \overline{\beta_i}^{-1}$ for $i \in S$
  and $\beta'_i = \beta_i$ for $i \in S^{c}$. If none of $\beta_1, \ldots \beta_N$
  lie on the unit circle in the complex plane, then by
  Lemma \ref{lem.rootsofthedot}, $x' \neq \dot{x}$ unless $S = \{1, \ldots , N\}$ meaning $|S| = N$. Hence
  if $0 < k < N$ then for a generic pair $(x,x') \in I_k$, $x' \neq x$ and
  $x' \neq \dot{x}$.
\end{proof}

  \subsection{Characterization of the components of $I$ in terms of convolution} \label{sec.convolution}
In \cite[Theorem 2.3]{BePl:15}, Beinert and Plonka prove that 
 two signals $x$ and $y$ have the same Fourier intensity function
  if and only if there exist finite signals $x_1, x_2$ such that $x = x_1 \star x_2$
  and $y= \lambda x_1 \star \dot{x_2}$ for some $\lambda \in S^1$.
  
  Their result can be made more precise by using our analysis of the
  irreducible components of the incidence variety.
  \begin{Theorem} \label{thm.convolution}
    The component $I_k \subset I$ parametrizes all pairs of equivalence classes
    $(x,x')$
    such that there exist vectors $x_1\in \C^{k+1}, x_2 \in \C^{N-k+1}$ such
    that $x = x_1 \star x_2$ and $x' = x_1 \star \dot{x_2}$.
  \end{Theorem}
  
    
  \begin{proof}
    If $x = x_1 \star x_2$ then $\hat{x}(\omega) = \hat{x_1}(\omega)
    \hat{x_2}(\omega)$.
    Thus if $\hat{x_1}(\omega) = a_0(\omega - \beta_1)\ldots
    (\omega - \beta_k)$ and
    $\hat{x_2}(\omega) = a'_0(\omega -\beta_{k+1})\ldots (\omega - \beta_{N-k})$
    then
    $\hat{x}(\omega) = (a_0a'_0 \omega^{-1}) (\omega - \beta_1) \ldots (\omega - \beta_N)$.

    Similarly if $x' = x_1 \star \dot{x}_2$ then
    $\hat{x'} = (\overline{\beta_{k+1}} \ldots \overline{\beta_{N-k}} )(\omega - \beta_1) \ldots (\omega - \beta_k)
    (\omega - \overline{\beta}^{-1}_{k+1}) \ldots (\omega - \overline{\beta}^{-1}_N)$. Hence $(x, x') \in I_k$. The converse is similar.
  \end{proof}
  \begin{Remark}
    Theorem \ref{thm.convolution} above says that we can identify $I_k$
    with the image of $\C^{k+1} \times \C^{N-k+1}$ under the map
    $(x_1, x_2) \mapsto \left((x_1 \star x_2), (x_1 \star \dot{x}_2)\right)$.
  \end{Remark}

  \section{Phase retrieval for vectors satisfying an algebraic condition}
  We can use our description of the incidence variety to prove that the generic vector satisfying any algebraic constraint can be uniquely recovered from its Fourier intensity function, provided there exists one such vector. Examples
  include vectors with a fixed entry or sparse vectors.
  This technique for multi-vectors
  played a crucial role in the paper \cite{BEE:18} on STFT.
      
\begin{Theorem}[Phase retrieval for vectors satisfying an algebraic condition] \label{Thm:GenericUniqueness}
      Let $W \subset X$ be a real subvariety of $X$ and suppose that
      there exists a point $w_0 \in W$ such that for all $(w_0,w_0')
      \in \pi^{-1}(w_0) \smallsetminus (I_0 \cup I_N$), $w_0' \notin
      W$ a generic $w \in W$ can be recovered up to global phase from
      its Fourier intensity function $|\hat{w}(\omega)|^2$.

    If the condition holds for all $(w_0,w_0') \in \pi^{-1}(w_0)
    \smallsetminus I_0$ then a generic $w \in W$ can be recovered up
    to trivial ambiguities.  Here $\pi \colon I \to X$ is the
    projection onto the first factor.
    \end{Theorem} \label{Thm.GenericUniqueness}
    \begin{proof}[Proof of Theorem]
  Let $I^0 = \overline{I \smallsetminus
      (I_0 \cup I_N)}$. Since $I^0$ is closed  the map
    $I^0 \to X$ is still finite. Let $I_W = I^0 \cap (W \times W)$
    be the real algebraic subset
    of $I^0$ consisting of pairs $(w,w')$ with $w,w'$ both in $W$.
    The image of $I_W$ under the projection $\pi \colon I \to X$
    is the set of $w \in W$ which cannot be recovered up to trivial ambiguity
    from $|\hat{w}(\omega)^2|$. We will show that $W \smallsetminus I_W$
    is Zariski dense.

    By
    assumption there exists $w_0 \in W$ such that for all pairs $(w_0,w_0') \in 
I^0$,
    $w_0' \notin W$. This implies that $W \times W$ intersects each irreducible
    component of $\pi^{-1}(W)$ in a proper algebraic subset. Hence,
    $\dim I_W < \dim \pi^{-1}(W) = \dim W$. 
    Thus,  $\dim \pi(I_W) < \dim W$ so $\pi(I_W)$ is contained in a proper algebraic subset of $W$. Hence
    the complement of $\pi(I_W)$ is dense in the real Zariski topology on
    $W$.
    \end{proof}

  \subsection{Imposing uniqueness with additional conditions}
  Using Theorem \ref{Thm:GenericUniqueness} we can show that a signal
  can be recovered modulo trivial ambiguities from the Fourier intensity
  function and the absolute value of a single entry. We illustrate
  with the following Corollary which is also proved in \cite[Corollary
    4.4]{BePl:15}. See the paper \cite{BePl:18} for more conditions which 
impose uniqueness.
  
  \begin{Corollary}\cite[Corollary 4.4]{BePl:15}
    For generic $x \in \C^\times \times \C^{N-1} \times \C^\times$
    the system of equations
\begin{eqnarray*}
    |\hat{x'}(\omega)|^2 & = &|\hat{x}(\omega)|^2\\
  |x'[N]| & = & |x[N]|
  \end{eqnarray*}
    has a unique solution modulo global phase. If $x'[N] = x[N]$
    then the solution is unique.
\end{Corollary}
  \begin{proof}
    Let $|x[N]| = a$ with $a > 0$.
    By Theorem \ref{Thm.GenericUniqueness} it suffices to find a single
    vector $x$ with $x[N] = a$ such that for all $(x,x') \in \pi^{-1}(x) \cap
    (I \smallsetminus I_0\cup I_N)$, $|x'[N]|  \neq a$.

    We do this as follows: Let $x = (a', 0, \ldots , 0, a)$ with $a' >0$
    and not equal to $a$.
The Fourier polynomial $\hx(\omega) = a' + a \omega^N$
    so $|\hx(\omega)^2| =  (a^2 + (a')^2) + (aa')\omega^N + (a'a) \omega^{-N}$.
    If $x' = (a_0, \ldots, a_N)$ has the same Fourier intensity function
    then $a_0\overline{a_N} =aa'$. If $|a_N| = a$
    then $|a_0| = a'$. But
    the constant coefficient is $|a_0|^2 + \ldots |a_N|^2 = a'^2  + a^2$
    so we conclude that all other entries in $x'$ are 0.
    Hence, up to global phase $x' = (a_0,0,  \ldots, 0 , a)$. But
    $a_0 a = a a'$ so $a_0 = a'$; ie $x' = x$.
  \end{proof}
 \subsection{Imposing uniqueness for multivectors}

 Theorem \ref{Thm:GenericUniqueness} can easily be generalized to multi-vectors.
 It is this form of the theorem that was used in \cite{BEE:18}.
 Given positive integers, $N_1, \ldots , N_m$
 let  $X[n] = \C^{N_n+1}/S^1$
Let $I[n] \subset X[n] \times X[n]$
be the incidence variety and let $\pi[n] \colon I[n] \to X[n]$ be the projection
to the first factor. Let $X = X[1] \times \ldots \times X[m]$ and $I =
I[1] \times \ldots I[m]$ be the product of the incidences. Finally let
$\pi \colon I \to X$ be the product of the projections $\pi[n]$.

\begin{Theorem}[Imposing uniqueness for multivectors] \label{thm.generalized}
  Let $W$ be an irreducible algebraic subset of $X$.
Suppose that there exists an $m$-tuple of vectors $w_0\in W$
such that for all $(w_0,w_0') \in \pi^{-1}(w)$, $w'_0$ is not obtained
from $w_0$ by a trivial ambiguity.
Then the generic $m$-tuple 
$w \in W$ can be recovered (up to phase) from the
Fourier intensity functions of its component vectors.\end{Theorem}

      \begin{proof}
        We use the same argument as in the proof of Theorem \ref{Thm:GenericUniqueness} to
        show that the set of $w \in W$ that cannot be recovered from their Fourier intensity function has strictly smaller dimension than $W$.
        \end{proof}
        
     \begin{Example}\cite[Proposition B.1]{BEE:18}
       In \cite{BEE:18} we consider the problem of giving lower bounds on the number of measurements required for blind phaseless STFT for signals
       of length $N$, windows of length $W$ and step size equal
       to $L$. The main result of that paper
       is that $\sim 10N$ measurements are sufficient for generic signal recovery modulo ambiguities and this is independent of the step size or window length.

       As part of the proof we need to show that a generic triple
       $(y_1, y_2, y_3)$ in the subvariety $Z \subset \C^{L+1} \times
       \C^{2L+1} \times \C^{3L+1}$ defined by the system of quadratic
       equations
       $$\{y_1[n] y_3[L+n] = y_2[n]y_2[L+n]\}_{n=0,\ldots , L}$$ is
       uniquely determined up to global phase by the Fourier intensity
       functions of the vectors $y_1, y_2, y_3$ . By Theorem
       \ref{thm.generalized} it suffices to explicitly demonstrate one
       triple $(y_1, y_2, y_3) \in Z$ which is uniquely determined by
       the Fourier intensity functions of the vectors $y_1,y_2,y_3$.
       \end{Example}

\def\cprime{$'$}
 \bibliographystyle{plain}

\end{document}